\documentclass[11pt,a4paper]{article}

\usepackage{moreverb,url}

\usepackage[colorlinks,bookmarksopen,bookmarksnumbered,citecolor=red,urlcolor=red]{hyperref}

\newcommand\BibTeX{{\rmfamily B\kern-.05em \textsc{i\kern-.025em b}\kern-.08em
T\kern-.1667em\lower.7ex\hbox{E\kern-.125emX}}}

\usepackage{graphicx, amsmath, amsthm, theorem, amssymb}

\usepackage{algorithm}
\usepackage{algorithmicx}
\usepackage[noend]{algpseudocode}
\usepackage{graphicx}
\usepackage{accents}

\usepackage{color}

\newcounter{theo}
\renewcommand{\thetheo}{\arabic{theo}}

\newenvironment{theo}
{%
\refstepcounter{theo}%
\par\medskip
\noindent\textbf{Theorem~\thetheo.}\itshape
}
{%
\par\medskip
}

\newcounter{lm}
\renewcommand{\thelm}{\arabic{lm}}

\newenvironment{lm}
{%
\refstepcounter{lm}%
\par\medskip
\noindent\textbf{Lemma~\thelm.}\itshape
}
{%
\par\medskip
}

\newcounter{rem}
\renewcommand{\therem}{\arabic{rem}}

\newenvironment{rem}
{%
\refstepcounter{rem}%
\par\medskip
\noindent\textbf{Remark~\therem.}\itshape
}
{%
\par\medskip
}

\newcommand{\OBS}{\mathcal{O}}
\newcommand{\diff}{d}
\newcommand{\uideg}{\textnormal{deg}w}
\newcommand{\mw}{q}
\newcommand{\tih}{\widetilde{h}}
\newcommand{\tiz}{\widetilde{z}}
\newcommand*{\dt}[1]{%
\accentset{\mbox{\large\bfseries .}}{#1}}
\newcommand{\Li}{\mathcal{L}}
\newcommand{\giat}{\widehat{g}}
\newcommand{\tit}{\widetilde{\theta}}
\newcommand{\UIE}{\mathcal{UIE}}
\newcommand{\sumj}{\sum{j=1}^{m}}
\newcommand{\sumk}{\sum_{k=m+1}^{\mw}}
\newcommand{\uidnf}{\mathcal{N}_{\tih}(\Sigma)}

\begin{document}

\title{Structural Foundations of Nonlinear Systems with Unknown Inputs: The UID-Induced Normal Form and Minimal-Sensing Structure-from-Motion}

\author{Agostino Martinelli\
INRIA Rhone Alpes, France\\
\texttt{agostino.martinelli@inria.fr}}

\date{}

\maketitle

\begin{abstract}
This paper establishes the first general structural solution to the problem of state estimation for nonlinear systems driven by unknown inputs.
Building upon nonlinear unknown-input observability theory, we show that every such system admits a structurally equivalent representation, referred to as the \emph{UID-induced normal form}.
The proposed representation decomposes the information carried by the unknown inputs into two complementary components: unknown-input directions that are structurally decoupled from the observable dynamics and observable quantities that completely represent the unknown-input information affecting the observable dynamics.
As a consequence, the UID-induced normal form provides a unified structural solution to unknown-input decoupling and unknown-input reconstruction, without requiring any model or stochastic assumption on the unknown inputs.
The practical significance of the proposed framework is demonstrated through a previously unexplored minimal Structure-from-Motion configuration.
The proposed representation enables recursive state estimation from only three point features and a single-axis gyroscope, allowing the recovery of the three-dimensional structure and camera motion up to an unknown global scale factor.
Experiments on real-world data validate the proposed framework and demonstrate the feasibility of this minimal sensing configuration.

\end{abstract}

\section{Introduction}\label{SectionIntroduction}

State estimation for nonlinear systems driven by unknown inputs remains one of the fundamental open problems in robotics and control.
Unknown inputs may arise from unmodeled disturbances, environmental interactions, actuator effects, or motion components that are not directly measured.
Unlike unknown parameters or uncertain model coefficients, unknown inputs explicitly enter the system dynamics and continuously affect the state evolution.
Consequently, although the system equations are known, the actual dynamics followed by the state are not.
This lack of knowledge fundamentally limits state estimation.
All estimation methods for dynamical systems, whether based on recursive filtering, nonlinear observers, smoothing, or batch optimization, rely on the system dynamics to relate the state at different time instants.
In recursive estimation, this dependence appears explicitly in the prediction step, where the state is propagated according to the system dynamics before incorporating new measurements.
Since this propagation requires knowledge of the unknown inputs, the prediction model cannot, in general, be directly evaluated.
This difficulty lies at the heart of state estimation for nonlinear systems driven by unknown inputs.

To overcome this fundamental difficulty, the literature has traditionally followed three distinct research directions.

The first direction (modeling-based and state-augmentation methods) circumvents the problem by introducing additional models for the unknown quantities affecting the dynamics.
When these quantities have known dynamics, such as calibration parameters, sensor biases, or other constant system properties, they can be incorporated into an augmented state vector and estimated jointly with the system state.
This approach is widely used in robotics for sensor calibration and multi-modal state estimation, where observability analysis has been exploited to characterize the identifiability of calibration parameters and inertial quantities \cite{MirzaeiRoumeliotis2008,JonesSoatto2011}.
More generally, visual--inertial estimation methods routinely augment the state with additional variables, such as inertial biases and calibration parameters, within filtering frameworks \cite{MSCKF,Hesch2014,OCVINS}.
For unknown inputs whose evolution is not available, however, state augmentation requires introducing an assumed dynamical model.
A representative example is MonoSLAM \cite{davison2007}, where unknown camera accelerations are modeled as zero-mean Gaussian process noise in the prediction step of an EKF.
More generally, these approaches replace the original estimation problem with a different one by postulating an evolution model for the unknown inputs.
Their performance therefore ultimately depends on the validity of the adopted assumptions.


The second direction (unknown-input decoupling methods) seeks instead to solve the original problem without modeling the unknown inputs by eliminating their influence from the estimation dynamics.
This idea underlies the extensive literature on Unknown Input Observers (UIOs) and disturbance decoupling.
In the linear setting, the corresponding conditions are typically expressed through algebraic rank conditions and invariant subspace techniques (e.g., \cite{houmuller1992}), whereas nonlinear extensions rely on differential-geometric constructions involving distributions, Lie derivatives, and coordinate transformations (e.g., \cite{isidori1990decoupling}).
Despite decades of research, available nonlinear results remain limited to specific classes of systems and provide only partial structural conditions rather than a general characterization applicable to arbitrary nonlinear dynamics.

The third direction (unknown-input reconstruction methods) also aims at solving the original problem without introducing models for the unknown inputs.
Rather than eliminating their influence, these approaches attempt to recover the unknown-input information directly from output measurements and their time derivatives, typically through algebraic elimination or recursive differentiation techniques (e.g., \cite{fliess2008nonlinearestimation}).
In robotics, observability-based approaches have also been exploited to reconstruct unknown motion components directly from sensory measurements. For instance, Belo et al. \cite{Belo2010} showed that, in planar bearing-only SLAM, three known landmarks are sufficient to reconstruct the robot motion and the unknown inputs without assuming a motion model.
Although successful in several specific applications, the available methods remain problem dependent and no general structural framework is available for nonlinear systems.

These three research directions reflect fundamentally different strategies for overcoming the lack of knowledge of the system dynamics induced by unknown inputs.
The first circumvents the problem through modeling assumptions, whereas the second and the third seek to solve it directly by, respectively, eliminating or reconstructing the unknown-input information without assuming any model for its evolution.

In parallel with these developments, observability has progressively evolved from an analysis tool into a fundamental design principle for state estimation in robotics and computer vision.
Rather than merely characterizing which state components are observable, observability analysis has increasingly guided the design of filtering, smoothing, and optimization-based estimation algorithms by identifying the correct observable and unobservable directions of the system.

Early seminal works in SLAM and cooperative localization established that estimator inconsistency is directly linked to the incorrect handling of the system's unobservable subspace, showing that preserving the correct observability properties is essential for consistent covariance propagation \cite{Huang2008,Huang2010,MSCKF,Hesch2014,OCVINS}.
Subsequent work further demonstrated that, in vision-aided inertial navigation, observability analysis must account for the geometry of indistinguishable trajectories, leading to refined notions of sensitivity and model reduction beyond the classical observable/unobservable dichotomy \cite{Soatto2013ModelReduction,Soatto2015Observability}.
Observability analysis has also enabled closed-form and algebraic solutions in sensor calibration \cite{TRO11} and visual--inertial sensor fusion \cite{TRO12}, and has subsequently guided the design of estimation algorithms for increasingly complex sensing systems, including visual--inertial and LiDAR--inertial navigation with heterogeneous geometric features \cite{Huang2019,LICFusion2020}.
More recently, observability has also been exploited in observability-aware trajectory planning \cite{ObservabilityAwareUAV} and to analyze the evolution of the unobservable subspace throughout practical visual--inertial estimation pipelines \cite{Tian2025USE}.

A common feature of all these developments is that they consider systems driven by known inputs, where the observability properties are entirely determined by the system geometry and the sensing configuration.
Whether observability can provide an equally general structural foundation for state estimation in the presence of unknown inputs has remained an open question.

This paper answers this question affirmatively.
We show that every nonlinear system with unknown inputs admits a structurally determined representation, referred to as the \emph{UID-induced normal form}, that provides the first general structural solution to the two fundamental problems that have driven unknown-input estimation research for decades: unknown-input decoupling and unknown-input reconstruction.
The proposed result builds upon the first general solution of the nonlinear unknown-input observability problem established in \cite{TAC19,IF22,TAC_UIO}.
The key observation is that the corresponding observability analysis reveals a structural decomposition of the unknown-input information, rather than merely determining the observable component of the state.
In particular, the recursive algorithms developed in \cite{IF22,TAC_UIO} generate a set of observable quantities, referred to as \emph{unknown input identifiers} (UIDs).
We show that, once an admissible set of UIDs has been selected, every nonlinear system with unknown inputs admits a uniquely determined UID-induced representation, referred to as the \emph{UID-induced normal form}.
Although the resulting representation depends on the selected set of UIDs, different admissible choices lead to equivalent UID-induced normal forms related by an invertible change of coordinates in the unknown-input space.
This representation holds for general nonlinear systems, including arbitrary nonlinear dependence on the state, known inputs, and unknown inputs (see Equation~(\ref{EqGeneralSystem})).
It is not an observer-dependent construction nor an ad hoc coordinate transformation, but an intrinsic structural property emerging from the observability analysis itself.

The UID-induced normal form decomposes the information carried by the unknown inputs into two complementary components.
The first consists of unknown-input directions that do not influence the evolution of any observable quantity and are therefore structurally decoupled from state estimation.
The second contains all unknown-input information that affects the observable dynamics.
Remarkably, this information is not represented through unknown variables but through observable quantities generated recursively by the observability analysis, and is therefore reconstructible from the system outputs and a finite number of their time derivatives.

This decomposition shows that unknown-input decoupling and unknown-input reconstruction are not independent problems, but two complementary manifestations of the same observability structure.
Unlike previous approaches, which provide decoupling conditions or reconstruction procedures only for restricted classes of nonlinear systems, the UID-induced normal form establishes a general framework in which both objectives emerge simultaneously from the same structural decomposition.
As a consequence, state estimation can be formulated without introducing dynamical models, stochastic assumptions, or independent estimators for the original unknown inputs.

We demonstrate the practical implications of the proposed framework through a Structure-from-Motion (SfM) problem, where the UID-induced normal form enables the design of a recursive estimator for a previously unexplored minimal sensing configuration.
Classical SfM seeks to recover, up to a global scale factor, both the three-dimensional structure of a scene and the motion of a moving monocular camera from image measurements.
Since the instantaneous camera motion has six degrees of freedom, corresponding to three linear and three angular velocity components, SfM naturally leads to estimation problems involving multiple unknown inputs.
In particular, when only a single gyroscope axis is available, five motion components remain unknown and can be treated as unknown inputs.

A fundamental result in multiple-view geometry is that three point features are sufficient for SfM only when the camera rotations are known, or equivalently when the relative orientations between camera poses are available.
In practice, this corresponds to having complete rotational information, for example from a three-axis gyroscope.
This setting is closely related to classical SfM formulations with known camera rotations, where structure reconstruction reduces to triangulation up to an unknown global scale factor \cite{HartleyZisserman}.
In the absence of rotational information, classical SfM requires a larger number of visual features.
In particular, at least four point correspondences are typically required \cite{HartleyZisserman}, while minimal two-view formulations require five point correspondences \cite{Nister2004FivePoint} and classical epipolar geometry estimation relies on the well-known eight-point algorithm \cite{LonguetHiggins1981}.

The nonlinear unknown-input observability analysis developed in \cite{TAC19,IF22,TAC_UIO} reveals a minimal SfM configuration in which the three-dimensional structure and camera motion can be recovered, up to an unknown global scale factor, from only three point features and a single-axis gyroscope.
In this setting, the five remaining motion components are treated as unknown inputs.
To the best of our knowledge, this minimal sensing configuration has not been reported in the computer vision literature\footnote{The reason why this configuration does not arise in classical minimal SfM formulations is that nonlinear observability analysis is formulated in continuous time, whereas classical SfM problems are defined from a finite number of discrete images. In the latter setting, finite rotations about a single measured axis generally do not provide sufficient constraints because of the noncommutativity of three-dimensional rotations. By contrast, observability analysis relies on infinitesimal motions, whose first-order evolution is independent of Lie-bracket terms associated with rotational noncommutativity. The identified configuration should therefore be interpreted as the continuous-time counterpart of classical SfM, becoming practically relevant when visual and inertial measurements are acquired at sufficiently high rates relative to the system dynamics.}.

In this minimal configuration, a standard recursive estimator cannot directly propagate the state because five components of the camera motion are unknown.
The UID-induced normal form resolves this difficulty by transforming the unknown-input dynamics into a representation where the observable subsystem evolves exclusively in terms of measurable quantities.
This enables the construction of an EKF-based recursive estimator without introducing any model for the unknown motion components.

Finally, we analyze more general sensing configurations involving multiple gyroscope axes as well as the purely vision-based case.
These correspond respectively to four, five, and six unknown motion components.
In particular, when four or more point features are available, the proposed framework can also be applied without gyroscope measurements, consistently with the observability properties predicted by classical SfM theory.
Experimental results on real-world data validate the proposed approach and demonstrate the practical feasibility of the identified minimal Structure-from-Motion configuration.

\subsection{Contributions}

The contributions of this paper are as follows:

\begin{enumerate}

\item
We provide the first general structural solution to the problem of state estimation for nonlinear systems driven by unknown inputs.
Specifically, we show that every nonlinear system with unknown inputs admits a uniquely determined representation, referred to as the \emph{UID-induced normal form}, obtained from the unknown-input observability structure.
This representation provides a complete decomposition of the information carried by the unknown inputs into two complementary components: a component that is structurally decoupled from the observable dynamics and a component that is entirely represented through observable quantities generated by the observability analysis.
As a consequence, unknown-input decoupling and unknown-input reconstruction are no longer separate problems, but emerge as two aspects of the same decomposition revealed by the observability analysis.

\item
We demonstrate the practical relevance of the proposed framework through a Structure-from-Motion application.
Using the UID-induced normal form, we construct a recursive EKF estimator for a minimal sensing configuration consisting of three point features and a single-axis gyroscope, where the remaining motion components are treated as unknown inputs.
The proposed formulation reveals a previously unexplored minimal SfM configuration and provides a practical estimation framework validated through experiments on real-world data.
Moreover, we analyze additional sensing configurations, including cases with multiple gyroscope axes and the purely vision-based setting, showing the consistency of the proposed framework with classical SfM observability results.

\item
For completeness and accessibility, we provide a concise description of the algebraic observability algorithm developed in \cite{TAC_UIO}, together with a MATLAB implementation that facilitates the application of the proposed framework \cite{TUIO}.

\end{enumerate}

\section{Mathematical characterization}\label{SectionProblemFormulation}

\subsection{System Class}\label{SubSectionSystem}

We consider nonlinear systems driven by both known control inputs and unknown inputs. Throughout this paper, the control inputs are assumed to be assigned functions of time and therefore regarded as known signals. This setting corresponds to the estimation problems typically encountered in robotics, where the robot executes a prescribed task and the control actions are available to the estimator.

The class of systems considered throughout this work is described by

\begin{equation}
\label{EquationSystemDefinitionUIOsimp}
\left\{
\begin{array}{ll}
\dot{x}
&=
g^0(x,t,u)
+
\displaystyle\sum_{j=1}^{\mw}
g^j(x,t,u)\,w_j,
\\[2mm]
y
&=
\left[
h_1(x,t,u),
\ldots,
h_p(x,t,u)
\right]^\top .
\end{array}
\right.
\end{equation}

where:

\begin{itemize}

\item $t\in I$, with $I\subseteq \mathbb{R}$ an interval of time;

\item $x\in \mathcal{M}_n$, where $\mathcal{M}_n$ is an analytic manifold of dimension $n$;

\item $u=[u_1,\ldots,u_{m_u}]^\top$ denotes the known control inputs, assumed to be smooth functions of time;

\item $w=[w_1,\ldots,w_{\mw}]^\top$ denotes the unknown inputs, assumed to be smooth functions of time;

\item $g^0,g^1,\ldots,g^{\mw}$ are analytic vector fields on $\mathcal{M}_n$;

\item $h_1,\ldots,h_p$ are analytic scalar functions.

\end{itemize}

System (\ref{EquationSystemDefinitionUIOsimp}) will be referred to as the \emph{working representation} throughout this paper. This form explicitly separates the contribution of each unknown-input direction through the vector fields $g^1,\ldots,g^{\mw}$. As will become apparent in the following sections, both the observability analysis and the UID-induced normal form are naturally formulated in terms of these vector fields and their interaction with the output functions.

A large class of systems arising in robotics, computer vision, navigation, and control is already naturally expressed in the form (\ref{EquationSystemDefinitionUIOsimp}). Although (\ref{EquationSystemDefinitionUIOsimp}) is affine with respect to the unknown inputs, it does not constitute a restriction. Indeed, any nonlinear system whose dynamics and outputs depend arbitrarily on unknown inputs can be transformed into the form (\ref{EquationSystemDefinitionUIOsimp}) through a finite-dimensional state augmentation.

Specifically, consider the general nonlinear system

\begin{equation}
\label{EqGeneralSystem}
\left\{
\begin{array}{ll}
\dot{X}
&=
f(X,t,u_1,\ldots,u_{m_u},W_1,\ldots,W_{\mw}),
\\[1mm]
y
&=
h(X,t,u_1,\ldots,u_{m_u},W_1,\ldots,W_{\mw}),
\end{array}
\right.
\end{equation}

where the dependence on the unknown inputs
$W_1,\ldots,W_{\mw}$
is completely arbitrary and not necessarily affine.
Define the extended state
\(
x=
\begin{bmatrix}
X \\
W
\end{bmatrix},
\)
where
\(
W=
\begin{bmatrix}
W_1 & \cdots & W_{\mw}
\end{bmatrix}^{\!\top},
\)
and introduce the new unknown inputs
\(
w_j=\dot{W}_j,
~
j=1,\ldots,\mw .
\)

The dynamics of the extended state become

\[
\dot{x}
=
\begin{bmatrix}
f(X,t,u,W)
\\
0
\end{bmatrix}
+
\sum_{j=1}^{\mw}
\begin{bmatrix}
0
\\
e_j
\end{bmatrix}
w_j,
\]

where $e_j$ denotes the $j$th vector of the canonical basis of $\mathbb{R}^{\mw}$.

Furthermore, the output can be rewritten as

\[
y=h(x,t,u),
\]

since the unknown inputs have been incorporated into the state.

Therefore, system (\ref{EqGeneralSystem}) admits an equivalent representation of the form (\ref{EquationSystemDefinitionUIOsimp}). Consequently, the analysis developed in this paper applies to arbitrary nonlinear systems driven by unknown inputs through their corresponding working representation.

In the remainder of the paper, all observability results, the construction of the unknown input identifiers (UIDs), and the derivation of the UID-induced normal form will be formulated for systems written in the form (\ref{EquationSystemDefinitionUIOsimp}).
\subsection{Active and Passive Unknown-Input Observability}\label{SubSectionActivePassiveObs}

The notion of observability considered throughout this paper is local with respect to the initial state, in accordance with the standard differential-geometric framework for nonlinear systems. Unless otherwise stated, all distinguishability properties are therefore understood with respect to sufficiently small neighborhoods of the initial condition.

For nonlinear systems driven by both known control inputs and unknown inputs, observability admits two complementary interpretations, which we refer to as \emph{active} and \emph{passive} observability. These two notions coincide in the absence of control inputs and differ only when the system is driven by inputs that can be assigned by the user.

We first consider the case in which both control inputs and unknown inputs are present. Let $u(\cdot)$ denote the known control inputs and $w(\cdot)$ the unknown inputs. In the \emph{active} interpretation, two initial states are said to be distinguishable if there exists at least one admissible control input function $u(\cdot)$ such that, for any admissible unknown input trajectory $w(\cdot)$, the corresponding output trajectories are different. A state is therefore said to be observable if it can be distinguished from all neighboring initial states in this sense.

In contrast, the \emph{passive} interpretation fixes the control inputs a priori. Observability is then defined with respect to a given control input signal. In this setting, the system is regarded as driven by known time functions, and observability characterizes which components of the state can be reconstructed from the outputs and their time derivatives independently of $w$. Consequently, observability in the passive sense depends explicitly on the chosen control input trajectory.

From a practical perspective, particularly in robotics, the passive interpretation is often the most relevant. In many applications, the control inputs are determined by a prescribed task, and the estimation problem is naturally formulated as: given these inputs, what part of the state can be inferred from available measurements?

In the special case where only unknown inputs are present, the two notions coincide. In this situation, there is no freedom of input selection, and observability reduces to an intrinsic property of the input-output map.

In earlier works such as \cite{TAC19,IF22}, observability was formulated in the active sense. In contrast, \cite{TAC_UIO} and the present paper adopt the passive viewpoint. In fact, when control inputs are present but fixed, they can be equivalently regarded as known time-dependent terms within the system dynamics, effectively reducing the analysis to the passive setting.

\section{Observability Analysis in the Presence of Unknown Inputs}
\label{SectionUIOSolution}

This section provides a concise overview of the general observability framework for nonlinear systems driven by unknown inputs developed in \cite{TAC_UIO}. Rather than reproducing the theoretical derivations and proofs, our objective is to provide the operational ingredients required to apply the observability algorithm reported in Algorithm~\ref{AlgoFull}. The complete mathematical development, together with the convergence proofs and all formal definitions, can be found in \cite{TAC19,IF22,TAC_UIO}.
The presentation given here is intended to make the implementation of the algorithm straightforward, either from first principles or by using the publicly available MATLAB toolbox in \cite{TUIO}.

\subsection{The Algorithm and Its Ingredients}\label{SubSectionAlgo}

\begin{algorithm}\caption{The algorithm that computes an observability codistribution ($\OBS$) and the highest UI degree of reconstructability ($m_H$) for general nonlinear systems after their expression in the working representation (\ref{EquationSystemDefinitionUIOsimp}).}

\begin{algorithmic}[1]

\State Set $\Sigma$ the system in (\ref{EquationSystemDefinitionUIOsimp}), and
$\Omega=\textnormal{span}\left\{\diff h_1,\ldots,\diff h_p\right\}$.\label{AlgoLineINIT}

\Loop
\label{AlgoLineLoopInit}

\If{$\uideg \left({\Omega}\right)==\mw$ }
 
{\bf break}.
\EndIf\label{AlgoLineConditionCanonicIf}

\State $m=\uideg \left({\Omega}\right)$.\label{AlgoLineuideg}

\State $[\tih_1,\ldots, \tih_m]=\mathcal{S}(\Sigma,~\Omega)$,~~
$\Sigma=\mathcal{R}(\Sigma,~\tih_1,\ldots, \tih_m)$.\label{AlgoLineSelectionAndReordering}
\State \textnormal{Compute} $\widehat{g}^\alpha$ (Eq \ref{Equationgalpham})).\label{AlgoLineMuNuGiat}

\State $\Omega_*=\left<\left. \dt{\Li}_{\giat^0}, \ldots, \Li_{\giat^m}~\right|\Omega\right>$.\label{AlgoLineOmegaStar}

\If{($\giat^k\notin\Omega_*^\bot)$, for a given $m+1\le k\le \mw$}\label{AlgoLineConditionOmegaStarIf}

\State Determine $\tit$ and set $\Sigma\triangleq\UIE\left(\Sigma,~\tit\right)$. \label{AlgoLineConditionOmegaStarNo}

\State \textnormal{Set} $\Omega\triangleq\Omega+\diff\tit+\sumj\sum_{\alpha_j=0}^{\overline{j}}\diff v^{(\alpha_j)}_j$.\label{AlgoLineReset}

\Else
\label{AlgoLineConditionOmegaStarElse}

\State $\OBS=\Omega_*$, $m_H=m$ {\bf return}.\label{AlgoLineConditionOmegaStarReturn}

\EndIf\label{AlgoLineConditionOmegaStarEnd}

\EndLoop\label{AlgoLineLoopFinit}

\State $[\widetilde{h}_1,\ldots, \widetilde{h}_{\mw}]=\mathcal{S}(\Sigma,~\Omega)$ \textnormal{and compute} $\widehat{g}^\alpha$.\label{AlgoLineFINALPrep}
\State $\OBS=\left<\dt{\Li}_{\giat^0},\Li_{\giat^1},\ldots,\Li_{\giat^{\mw}}~\left|~\Omega\right.\right>$, $m_H=\mw$.\label{AlgoLineFINALSTEP}

\end{algorithmic}
\label{AlgoFull}
\end{algorithm}

Algorithm~\ref{AlgoFull} computes two fundamental quantities associated with the nonlinear system (\ref{EquationSystemDefinitionUIOsimp}):

\begin{itemize}

\item
the observable codistribution, denoted by $\OBS$, which completely characterizes the observability properties of the system;

\item
the highest unknown input degree of reconstructability, denoted by $m_H$, as defined in \cite{TAC_UIO}; this notion will be briefly recalled later together with its intuitive interpretation.

\end{itemize}

The rigorous definitions of all the quantities appearing in the algorithm, together with the proof of its correctness and convergence, are given in \cite{TAC_UIO}. In the remainder of this section we only introduce the concepts that are necessary to understand and implement the algorithm.

The first concept is that of a codistribution.
A codistribution on the state manifold $\mathcal{M}_n$ is a vector space generated by a collection of row vectors belonging to the cotangent space of $\mathcal{M}_n$. Since these row vectors generally depend on the state $x\in\mathcal{M}_n$, the corresponding vector space also varies with the state. Throughout this paper all the functions defining the system are assumed to be analytic. Under this assumption, the codistributions generated during the algorithm satisfy the standard regularity properties of geometric control theory and, in particular, have constant dimension over $\mathcal{M}_n$.

The special case of interest is that in which the generators are differentials of scalar functions. Such codistributions are called {\it integrable} and are the only ones generated by Algorithm~\ref{AlgoFull}. If $\Lambda$ is generated by the differentials of the scalar functions
$\lambda_1,\ldots,\lambda_r$, we write

\[
\Lambda=\mathrm{span}\{\diff\lambda_1,\ldots,\diff\lambda_r\}.
\]

The algorithm proceeds recursively. It starts from the codistribution generated by the output differentials,

\[
\Omega=\mathrm{span}\{\diff h_1,\ldots,\diff h_p\},
\]

as shown in Line~\ref{AlgoLineINIT}, and recursively enlarges this codistribution until the observable codistribution $\OBS$ is obtained.

A central quantity used throughout the algorithm is the unknown-input degree of reconstructability from a given codistribution, denoted by

\[
\uideg(\Lambda).
\]

Its formal definition is given in \cite{IF22, TAC_UIO}; operationally, it is computed as follows.
Assume that the integrable codistribution $\Lambda$ has dimension $r$ and is generated by the differentials of the scalar functions

\[
\lambda_1,\ldots,\lambda_r .
\]

One constructs the 
$r\times \mw$
matrix whose $(i,j)$ element is the Lie derivative

\[
\Li_{g^j}\lambda_i ,
\]

where $g^j$ denotes the vector field associated with the $j$-th unknown input (see Eq. (\ref{EquationSystemDefinitionUIOsimp})).
The rank of this matrix is defined as the unknown-input degree of reconstructability of the system from $\Lambda$,

\[
m=\uideg(\Lambda).
\]

In addition, we define the highest unknown input degree of reconstructability as the unknown input degree of reconstructability from the entire observability codistribution, i.e., 

\[
m_H=\uideg(\OBS).
\]

By construction, for any codistribution $\Lambda=\mathrm{span}\{\diff\lambda_1,\ldots,\diff\lambda_r\}$, $m\leq r$ and $m\leq \mw$.
Therefore, we can always extract $m$ scalar functions among
$\lambda_1,\ldots,\lambda_r$ and $m$ vectors among $g^1,\ldots,g^\mw$, such that the corresponding $m\times m$ submatrix has full rank.

This extraction is precisely the role of the operator

\[
\mathcal{S}(\Sigma,\Omega),
\]

appearing in Line~\ref{AlgoLineSelectionAndReordering}. The selected functions,

\[
\widetilde h_1,\ldots,\widetilde h_m,
\]

are called \emph{unknown input identifiers} (UIDs). 

Each UID induces a new quantity defined as
\begin{equation}\label{Equationvj}
v_j = \dot{\widetilde h}_j,
\end{equation}

which will replace the corresponding unknown input $w_j$ in the system representation. Hence, UID directions are not estimated as unknown inputs but are instead reformulated as input-like quantities.

At the same time, the operator
$\mathcal R(\Sigma,\widetilde h_1,\ldots,\widetilde h_m)$
reorders the unknown inputs so that the matrix

\[
\left[\Li_{g^j}\widetilde h_i\right]_{i,j=1}^{m}
\]

is nonsingular.

Once the UIDs have been selected, the matrices $\mu$ and $\nu=\mu^{-1}$ are computed according to

\begin{equation}\label{EquationMu}
\mu^\beta_\alpha=
\left[
\begin{array}{ll}
\dt{\Li}_{g^0}\widetilde h_\alpha,
& \beta=0,\\
\Li_{g^\beta}\widetilde h_\alpha,
& \beta=1,\ldots,m,
\end{array}
\right.~~~\alpha=0,1,\ldots,m,
\end{equation}

where

\[
\dt{\Li}_{g^0}\triangleq
\Li_{g^0}+\frac{\partial}{\partial t},
\]
and we set\footnote{Setting $\tilde{h}_0 = t$ reflects a deeper conceptual point that can be explained as follows. Time influences a dynamical system in two distinct ways. On one hand, it behaves like an uncontrollable input  --  a property it shares with unknown inputs. On the other hand, it also acts as an output, since sensors typically provide both the measurement and the time at which it is taken. A powerful approach to modeling dynamical systems is to treat time as an unknown input. However, since time is actually known, we already possess a corresponding UI identifier  --  namely, \( t \) itself. Therefore, the choice of setting \( \tilde{h}_0 = t \) is both valid and conceptually elegant.} ~~$\tih_0\triangleq t$.

These quantities are then used to construct the vector fields
$\widehat g^\alpha$ defined in (\ref{Equationgalpham}), as indicated in Line~\ref{AlgoLineMuNuGiat}.

\begin{equation}\label{Equationgalpham}
   \giat^\alpha  = 
   \left[
   \begin{array}{ll}
   \sum_{\beta=0}^m~\nu^{\alpha}_{\beta} ~g^{\beta} & \alpha=0,1,\ldots,m\\
   g^\alpha-\sum_{j=1}^m\sum_{l=1}^m~\nu^j_{l}g^l\Li_{g^\alpha}\tih_j & \alpha=m+1,..,\mw\\
   \end{array}
   \right.
\end{equation}

%
%
%

The next step consists in computing the smallest codistribution containing $\Omega$ that is invariant under

\[
\dt{\Li}_{\widehat g^0},
\Li_{\widehat g^1},
\ldots,
\Li_{\widehat g^m},
\]

namely

\[
\Omega_*=
\left<
\dt{\Li}_{\widehat g^0},
\Li_{\widehat g^1},
\ldots,
\Li_{\widehat g^m}
\left|
\Omega
\right.
\right>,
\]

as shown in Line~\ref{AlgoLineOmegaStar}.

The convergence test is performed by evaluating whether every remaining unknown-input vector field
$\giat^k$, $k=m+1,\ldots,\mw$, belongs to the orthogonal complement of $\Omega_*$.
If this condition is satisfied, the recursion terminates. In this case,

\[
\OBS=\Omega_*,
\]

and the current value of $m$ coincides with the highest unknown-input degree of reconstructability,

\[
m_H=m.
\]

Otherwise, at least one additional unknown-input direction admits a UID representation, and can therefore be associated with an observable input of the form $v_j = \dot{\widetilde h}_j$.
Specifically, this is obtained by executing Lines~\ref{AlgoLineConditionOmegaStarNo}--\ref{AlgoLineReset}. A new scalar quantity $\tit$ is first constructed (see Eq.~(18) in \cite{TAC_UIO}), after which the system is extended according to the procedure described in Section~IV-H of \cite{TAC_UIO}. Additional differentials of new observable scalar functions are then incorporated into $\Omega$.
Theorem~1 of \cite{TAC_UIO} proves that this extension necessarily increases the unknown-input degree of reconstructability by one. Consequently, each iteration produces at least one additional UID. Since the number of unknown inputs is finite, the recursion always terminates after a finite number of iterations.

The algorithm may converge before all unknown-input directions admit a UID representation, yielding $m_H < \mw$. In this case, only $m_H$ unknown-input directions are replaced by the observable inputs of the form $v_j = \dot{\widetilde h}_j$, while the remaining directions do not admit a UID representation. As shown in Section \ref{SectionUIDNormalForm}, these remaining directions do not affect the observable subsystem.

\section{The UID-Induced Normal Form}
\label{SectionUIDNormalForm}
This section shows that Algorithm~\ref{AlgoFull} does not merely compute the observability codistribution, but also provides the structural ingredients for a canonical representation of nonlinear systems driven by unknown inputs.
We analyze the properties of this representation and show that the UID-induced normal form is obtained by an invertible transformation of the unknown-input space defined by the selected unknown input identifiers (UIDs).
In this representation, the information carried by the unknown inputs is decomposed, through the observability structure, into components that are structurally decoupled from the observable dynamics and components that are observable and reconstructible.
This decomposition provides a unified structural foundation for unknown-input decoupling and unknown-input reconstruction, showing that both problems arise from the same observability-based representation.

\subsection{From Algorithm~1 to a Structural Representation}
\label{SubSectionUIDNormalForm}

Given a nonlinear system $\Sigma$, Algorithm~\ref{AlgoFull} computes both the observable codistribution $\OBS$ and the integer $m_H$, which characterize the observable structure and the number of unknown-input directions affecting the observable dynamics.
These quantities are intrinsic structural properties of the system. Although their computation depends on the chosen state coordinates, the resulting codistribution is a geometric object and therefore represents the same observability properties under any smooth change of coordinates. Likewise, the value of $m_H$ is independent of the choice of coordinates.

We henceforth refer to all quantities generated by Algorithm~\ref{AlgoFull} at convergence. In particular, $m=m_H$, and the matrix $\mu$ in~\eqref{EquationMu} has dimension $(1+m_H)\times(1+m_H)$.

Let
\[
\tih=[\tih_1,\ldots,\tih_{m_H}]
\]
be an admissible selection of unknown input identifiers (UIDs). Such a selection is, in general, not unique. One admissible choice is automatically generated by Algorithm~\ref{AlgoFull}, although any admissible selection may be adopted. Specifically, Algorithm~\ref{AlgoFull} recursively selects the UIDs under the sole requirement that the matrix $\mu$ in~\eqref{EquationMu} remains nonsingular at each step.

Once an admissible UID selection $\tih$ has been fixed, the corresponding UID-induced normal form is uniquely determined by that choice. We denote this representation by

\[
\uidnf.
\]

Using the definitions of $\mu$ (see Eq.~(\ref{EquationMu})), $\nu=\mu^{-1}$, and the vector fields
$\widehat g^\alpha$ (see Eq.~(\ref{Equationgalpham})), a straightforward calculation yields the following
fundamental identity:

\begin{equation}\label{EquationReplacing}
g^0
+
\displaystyle\sum_{j=1}^{\mw}
g^j\,w_j
=
\giat^0
+
\displaystyle\sumj
\giat^j\,v_j
+\sumk\giat^k w^k.
\end{equation}

By exploiting this identity upon convergence of Algorithm~\ref{AlgoFull} (i.e., when $m=m_H$), the system dynamics in the UID-induced normal form are given by

\begin{equation}\label{EquationUIDNF}
\dot{x}
=
\giat^0
+
\sum_{j=1}^{m_H}\giat^j v_j
+
\sum_{k=m_H+1}^{\mw}\giat^k w_k,
\end{equation}

where the unknown-input directions associated with the selected UIDs are represented through the new coordinates
$v_1,\ldots,v_{m_H}$ defined in~\eqref{Equationvj}.
It is important to emphasize that this replacement does not simply relabel the first $m_H$ unknown inputs. The quantities $v_1,\ldots,v_{m_H}$ are instead defined by the following invertible relations (Equation~(\ref{Equationvj})):
\begin{equation}\label{Equationvi}
v_\alpha~=~\dt{\tih}_\alpha~=~
\frac{\partial\tih_\alpha}{\partial x}\cdot
\left(
g^0+\sum_{j=1}^{\mw} g^j~w_j
\right)~+ \frac{\partial\tih_\alpha}{\partial t}=
\end{equation}
\[
\sum_{\beta=0}^{m_H} \mu^\beta_\alpha~w_\beta+\sum_{k=m_H+1}^{\mw} \Li_{g^k}\tih_\alpha~w_k,
\]
where we used the dynamics in (\ref{EquationSystemDefinitionUIOsimp}), 
$\mu$ is given in (\ref{EquationMu}), and we set $v_0=w_0=1$. Hence, each $v_\alpha$ generally depends on all the original unknown inputs rather than only on $w_1,\ldots,w_{m_H}$. Therefore, the transformation defining the UID-induced normal form is a coordinate transformation in the unknown-input space rather than a relabeling of a subset of the original inputs:
\[
(w_1,\ldots,w_\mw)
\quad\longrightarrow\quad
(v_1,\ldots,v_{m_H},w_{m_H+1},\ldots,w_{\mw}),
\]
where only the last $\mw-m_H$ coordinates remain unchanged. The structural properties established in the following subsection therefore refer to this transformed representation rather than to the original unknown inputs.

The UID-induced normal form is therefore not an observer-dependent construction nor an ad hoc coordinate transformation.
Rather, it is a structural representation obtained from the unknown-input observability analysis and uniquely determined once an admissible set of unknown input identifiers has been selected.
Different admissible UID selections may lead to different coordinate representations of the unknown-input space; however, all such representations are equivalent with respect to the observable dynamics, as will be established through the results of Theorems~\ref{PropertyIrrilevance} and~\ref{PropertyReconstructability} in Section \ref{SubSectionUIDNFProperties}.
%

\subsection{Properties of $\uidnf$}
\label{SubSectionUIDNFProperties}

\subsubsection{Irrelevance of $w_k$, $k>m_H$:}

We show that, when the system is expressed in the UID-induced normal form, the variables $w_k$, 
$k>m_H$, no longer appear explicitly in the evolution of observable quantities. Their influence is 
entirely encoded in the variables $v_1,\ldots,v_{m_H}$. It is important to emphasize that these 
inputs correspond to the original unknown inputs only up to the reordering determined by the selection 
procedure $\mathcal{R}(\Sigma,\tih)$ in line~\ref{AlgoLineSelectionAndReordering}; hence, their 
indexing does not necessarily match the original ordering.

We establish the following fundamental result.

\begin{theo}\label{PropertyIrrilevance}
Consider the system expressed in the UID-induced normal form~\eqref{EquationUIDNF}.
Then, for every smooth function $\phi=\phi(x,t)$ whose differential belongs to the
observable codistribution, i.e., $\diff\phi\in\OBS$, the explicit dependence of
$\dot{\phi}$ on the variables $w_k$, $k>m_H$, vanishes. Equivalently, the time
derivative of every observable function depends only on the known inputs and on
the variables $v_1,\ldots,v_{m_H}$.
\end{theo}

\begin{proof}
The proof is given in Appendix \ref{AppendixIrrilevance}.
\end{proof}

\begin{rem}
The previous property is specific to the UID-induced normal form.
Indeed, the variables $v_1,\ldots,v_{m_H}$ are defined by the invertible transformation~\eqref{Equationvi} and generally depend on all the original unknown inputs, including
$w_k$, $k>m_H$.
Consequently, the theorem does not state that these original unknown inputs have no influence on the observable dynamics. Rather, it states that, after rewriting the system in the UID-induced normal form, their influence is entirely captured by the variables
$v_1,\ldots,v_{m_H}$, while the variables
$w_k$, $k>m_H$, no longer appear explicitly in the evolution of observable quantities.
\end{rem}

This result provides a complete structural characterization of unknown-input decoupling for the general class of nonlinear systems considered in this paper. Classical approaches typically derive sufficient conditions for constructing observers whose estimation error is insensitive to unknown inputs, or address only restricted classes of nonlinear systems. In contrast, Theorem~\ref{PropertyIrrilevance} establishes an intrinsic property of the system itself. Specifically, once the system is expressed in the UID-induced normal form~\eqref{EquationUIDNF}, every unknown-input direction represented by the variables $w_k$, $k>m_H$, is structurally decoupled from the evolution of all observable quantities. Equivalently, every unknown-input direction that influences the observable dynamics is necessarily represented through the observable quantities $v_j$.
Accordingly, every unknown-input direction that can be structurally eliminated from the observable dynamics has already been eliminated by the UID-induced normal form. The decomposition obtained from the UID-induced normal form is therefore complete with respect to the observable dynamics: the evolution of any observable function depends exclusively on the known inputs and on the observable quantities $v_j$, while the remaining unknown-input directions have no influence on the evolution of any observable quantity.
The UID-induced normal form therefore provides the explicit structural representation in which this decomposition is realized, revealing the intrinsic decoupling structure encoded in the observable codistribution.
Rather than introducing additional assumptions on the unknown inputs or relying on observer-specific constructions, the decoupling structure emerges directly from the observable codistribution through the observability analysis itself.

\subsubsection{Observability of $v$:}

We establish the following property.

\begin{theo}\label{PropertyReconstructability}
The quantities $v_1,\ldots,v_{m_H}$ are reconstructible from the system outputs and their time derivatives.
\end{theo}

\begin{proof}
The proof is given in Appendix \ref{AppendixReconstructability}.
\end{proof}

This result shows that the quantities $v_j$ are not independent unknown inputs in the UID-induced normal form. Instead, they are observable signals generated recursively by the observability analysis and correspond to the time derivatives of the observable quantities $\tilde h_j$.

Theorem~\ref{PropertyReconstructability} also provides a structural interpretation of unknown-input reconstruction. Existing approaches typically seek to recover unknown inputs from output measurements and their time derivatives through problem-dependent algebraic manipulations or recursive differentiation procedures, without a general criterion identifying which unknown-input information is actually relevant to state estimation. By contrast, the UID-induced normal form determines this information directly from the observability structure of the system. It identifies the complete unknown-input information that necessarily remains coupled to the observable dynamics and represents it through the observable quantities $v_1,\ldots,v_{m_H}$ defined in~\eqref{Equationvi}.

Moreover, Theorem~\ref{PropertyReconstructability} guarantees that these quantities are always reconstructible from the system outputs and their time derivatives, independently of whether the complete system state is observable. Consequently, the UID-induced normal form not only identifies exactly which unknown-input information is fundamentally required for the evolution of the observable subsystem, but also provides a constructive and systematic mechanism for obtaining that information. Unknown-input reconstruction therefore emerges as an intrinsic consequence of the observability analysis itself, rather than as a separate problem requiring ad hoc reconstruction procedures.

\subsection{General implications for state estimation}

The structural properties established in the previous subsection are independent of the particular state estimation methodology adopted. They follow directly from the UID-induced normal form and therefore apply to any state estimation framework, including recursive filtering methods (e.g., Extended Kalman Filters, Unscented Kalman Filters, particle filters), moving-horizon estimation, smoothing techniques, and batch optimization methods.
In other words, the UID-induced normal form characterizes which unknown inputs are structurally irrelevant for estimating the observable component of the state and which ones can instead be replaced by observable dynamical quantities. These properties hold irrespective of how the available measurements are processed by the estimation algorithm.

In the remainder of this subsection we discuss these general implications. Section~V then specializes these results to recursive state estimation by illustrating how the UID-induced normal form can be exploited within a recursive filtering framework.

Theorem~\ref{PropertyIrrilevance} establishes that the unknown inputs $w_k$, for $k>m_H$, do not affect the evolution of any observable quantity. As a consequence, these inputs are structurally irrelevant for state estimation and can be set to zero without loss of generality. Under this simplification, the dynamics of the UID-induced normal form~\eqref{EquationUIDNF} reduce to
\begin{equation}
\label{EquationUIDNFsimplified}
\dot{x}
=
\giat^0
+
\sum_{j=1}^{m_H}\giat^j v_j.
\end{equation}

Theorem~\ref{PropertyReconstructability} further implies that the variables $v_1,\ldots,v_{m_H}$ are reconstructible from the system outputs and their time derivatives. Therefore, the effective state estimation problem can be reformulated in terms of the augmented state $[x^\top, v^\top]^\top$, where the unknown inputs that influence the observable dynamics are replaced by observable quantities.

It is important to emphasize a structural distinction that arises in this framework. If the full state $x$ is observable, then, by Theorem~\ref{PropertyIrrilevance}, the vector fields $\giat^k$, for $k>m_H$, directly vanish.

In the general case, however, the state $x$ may not be fully observable. This can be directly inferred from the dimension of the observable codistribution $\OBS$. When $\dim(\OBS)$ is strictly smaller than the state dimension, a reduced observable substate should be identified. In principle, any basis of $\OBS$ defines such a substate. However, the corresponding basis functions may have highly nonlinear and non-structured expressions in the original coordinates, which makes it difficult to explicitly rewrite the system dynamics in a form analogous to~\eqref{EquationSystemDefinitionUIOsimp} or~\eqref{EqGeneralSystem}, where all quantities must be expressed solely in terms of the chosen substate. Although such a representation is always theoretically possible, no general constructive procedure exists for obtaining it in practice.

If such a representation can be derived, Algorithm~\ref{AlgoFull} can then be applied to the reduced system, leading to a corresponding UID-induced normal form once an admissible UID selection has been specified.

A further practical difficulty concerns the selection of the UIDs. While, from an observability perspective, any admissible selection is equivalent as long as the matrix $\mu$ remains nonsingular, the choice has in general an impact on state estimation performance. In particular, different UID selections may lead to different numerical conditioning, different sensitivity to noise, and different estimation accuracy for different components of the state. This suggests that the UID selection problem is inherently tied to the estimation method and the specific performance criteria of interest. We believe that no general, fully automated solution to this selection problem exists. It is worth noting that even in the classical setting without unknown inputs, the problem of identifying optimal coordinates for state estimation does not admit a general solution.

The discussion above is intentionally independent of any specific state estimation algorithm. The UID-induced normal form should be regarded as a structural representation of the original system that explicitly separates the unknown inputs that are irrelevant to the observable dynamics from those that can be replaced by observable dynamical quantities. Once this representation has been obtained, any state estimation methodology can be applied to the resulting system.

In the next sections we focus on recursive state estimation. We illustrate how the UID-induced normal form can be exploited within a recursive filtering framework through the proposed Structure-from-Motion application.

\section{Recursive State Estimation Based on the UID-Induced Normal Form}
\label{SectionRecursiveEstimation}

The structural results established in the previous section naturally suggest a general framework for recursive state estimation.
In this section, we illustrate this framework through an Extended Kalman Filter (EKF), adopted as a representative recursive state estimation algorithm.
The objective is to show how the UID-induced normal form can be directly incorporated into a standard recursive filtering scheme.

\subsection{Prediction model}

The prediction model derived from the UID-induced normal form assumes the availability of the increments of the unknown input identifiers,
\[
\Delta \tih_j,\qquad j=1,\ldots,m_H.
\]

Indeed, from~\eqref{EquationUIDNFsimplified}, the state propagation over the interval $\Delta t$ is given by

\begin{equation}
\label{EquationDeltaX}
\Delta x
=
\giat^{0}\Delta t
+
\sum_{j=1}^{m_H}\giat^{j}\Delta\tih_j.
\end{equation}

It is important to emphasize that the quantities
$\tih_j$
are observable, but are not necessarily directly measured outputs.
In the general case, each UID can be expressed as a function of the outputs and of a finite number of their time derivatives.
Consequently, before applying the above prediction model, the selected UIDs must first be reconstructed from the available measurements.
This reconstruction problem is conceptually independent of the UID-induced normal form.
The latter provides the structural prediction model, whereas the computation of the UIDs can be carried out by any suitable output-differentiation or state estimation technique.
Different reconstruction methods can therefore be combined with the same UID-induced normal form without modifying the resulting prediction equations.
In this paper we restrict our attention to the important class of systems for which the selected UIDs are directly measured outputs.
In this case, the increments
$\Delta\tih_j$
are immediately available from consecutive measurements, and the prediction model can be implemented directly.
This assumption is satisfied by the Structure-from-Motion problem considered in the next section and allows us to focus on the implications of the UID-induced normal form for recursive state estimation without introducing an additional UID reconstruction layer.

\subsection{UID selection}
\label{SubSectionSelectionEuristic}

The recursive estimation framework described in this section should be regarded as heuristic.
This is consistent with the nature of the Extended Kalman Filter itself, which is a heuristic extension of the Kalman filter to nonlinear systems.
Consequently, no general optimality criterion can be expected for selecting the unknown input identifiers.

Within the EKF framework, the choice of the UIDs determines how the available output measurements are exploited.
The selected UIDs are employed during the prediction step through the UID-induced normal form, while the remaining outputs are naturally available for the measurement update.
At the same time, the selected UIDs determine the subset of unknown-input directions that is replaced by the observable quantities
$v_j=\dot{\tih}_j$,
whereas the remaining unknown-input directions are discarded according to Theorem~\ref{PropertyIrrilevance}.
The only structural requirement is that the corresponding matrix $\mu$ remains nonsingular.

As a practical heuristic, among all admissible UID selections we choose the one that maximizes the absolute value of
$\det(\mu)$,
thereby favoring numerically well-conditioned transformations.
This criterion implicitly assumes that all output measurements are affected by comparable noise levels.
If different outputs exhibit significantly different noise characteristics, a weighted criterion could be adopted instead.

If the number of admissible selections,
\[
\binom{K}{m_H},
\]
is sufficiently small, where $K$ denotes the number of candidates, the optimum is obtained by exhaustive search.
Otherwise, we employ a simple greedy strategy.
The procedure is initialized by selecting the rows of $\mu$ through a QR factorization with column pivoting.
It then iteratively attempts to replace one selected row with one non-selected row whenever this increases
$|\det(\mu)|$.
The process terminates when no improving swap can be found.

It is worth emphasizing that the UID selection is generally state dependent.
Indeed, all the quantities defining the matrix $\mu$ ultimately depend on the system state through the vector and scalar fields appearing in~\eqref{EquationSystemDefinitionUIOsimp}.
As a consequence, the UID set that maximizes $|\det(\mu)|$ may vary as the state evolves.
Accordingly, the UID selection procedure is repeated at every prediction step, allowing the filter to adapt its structural representation to the current operating conditions.

\subsection{Recursive filtering}\label{SubSectionRecursiveFiltering}

Once an admissible UID set has been selected, the prediction and correction models are completely specified.
In particular, the selected UIDs determine the partition of the available measurements into the subset
$\tiz$,
used by the prediction model, and the complementary subset
$\bar z$,
used during the correction step.

The Extended Kalman Filter can therefore be implemented in the standard prediction--correction form.
The only quantities that differ from a conventional EKF are those determined by the UID-induced normal form and the selected UID set, as detailed below.

\subsubsection{State-transition Jacobian:}

The state-transition Jacobian propagates the state-estimation error through the prediction step.
Denoting it by $F_x$, from~\eqref{EquationDeltaX} we obtain

\[
F_x
=
I_n
+
\frac{\partial\giat^0}{\partial x}\Delta t
+
\sum_{j=1}^{m_H}
\frac{\partial\giat^j}{\partial x}
\Delta\tiz_j,
\]

where $I_n$ denotes the identity matrix of dimension equal to the state dimension, and
$\Delta\tiz_j$
is the difference between two consecutive measurements of the output selected as the $j$-th UID.

\subsubsection{Input Jacobian:}

The input Jacobian propagates the uncertainty associated with the inputs through the prediction model.
When known control inputs are present, this uncertainty includes both the unknown-input increments introduced by the UID-induced normal form and the uncertainty affecting the measured control inputs.

The corresponding Jacobian is

\[
F_u
=
\left[
\giat^1,\,
\ldots,\,
\giat^{m_H},
\;
\frac{\partial\giat^0}{\partial u}\Delta t
+
\sum_{j=1}^{m_H}
\frac{\partial\giat^j}{\partial u}
\Delta\tiz_j
\right].
\]

The first $m_H$ columns correspond to the uncertainty associated with the UID increments, whereas the remaining columns account for the uncertainty affecting the known control inputs.

\subsubsection{Process noise covariance:}

Assuming that the uncertainty affecting the known control inputs is statistically independent of the measurement uncertainty associated with the outputs, the process noise covariance matrix is given by

\[
Q
=
\begin{bmatrix}
Q_{\Delta\tiz} & 0\\
0 & Q_u
\end{bmatrix},
\]

where
$Q_{\Delta\tiz}$
denotes the covariance matrix associated with the uncertainty on the UID increments
$\Delta\tiz$,
and
$Q_u$
is the covariance matrix associated with the known control inputs.

Furthermore, if the measurement errors are temporally uncorrelated, the covariance of the measurement increments satisfies

\[
Q_{\Delta\tiz}
=
2Q_{\tiz},
\]

where
$Q_{\tiz}$
is the covariance matrix of the selected output measurements.

\subsubsection{Measurement Jacobian and measurement noise covariance:}

The correction step is identical to that of a standard EKF, except that it only exploits the output measurements that have not been selected as UIDs.
Accordingly, the original observation function is partitioned as

\[
h(x)=
\begin{bmatrix}
\tih(x)\\
\bar h(x)
\end{bmatrix},
\]

where
$\tih$
contains the selected UIDs and
$\bar h$
contains the remaining outputs.
The measurement Jacobian is therefore

\[
H
=
\frac{\partial\bar h}{\partial x},
\]

while the corresponding measurement noise covariance is simply the covariance matrix
$R_{\bar z}$
associated with the measurements
$\bar z$.

\section{Structure-from-Motion as a Nonlinear Unknown-Input Estimation Problem}
\label{SectionSfM}

This section applies the proposed UID-induced normal form framework to a minimal Structure-from-Motion (SfM) problem. 
The objective is to formulate the SfM problem as a nonlinear state estimation problem in the presence of unknown inputs, where the unknown inputs correspond to the camera motion components that are not directly measured.

Consider a monocular camera observing \(K\) static point features whose positions are fixed in the global reference frame. 
The camera motion is characterized by a linear velocity and an angular velocity expressed in the camera local frame. 
When gyroscopic measurements are not available, all six motion components are unknown inputs: three linear velocity components and three angular velocity components.
More generally, when only partial rotational information is available, the remaining motion components are treated as unknown inputs within the proposed framework.

\subsection{State representation}\label{SubSectionSfMState}

The state is expressed in the camera local reference frame. 
For each point feature, instead of directly considering its three-dimensional coordinates, we adopt the inverse-depth parametrization commonly used in monocular Visual SLAM \cite{Civera2008}. 
Specifically, each feature is represented by its bearing direction and inverse depth.

Let
\(
[F_x^i,F_y^i,F_z^i]^\top
\)
denote the coordinates of the \(i\)-th feature point expressed in the camera frame.
The corresponding unit bearing vector is defined as

\[
[u_x^i,u_y^i,u_z^i]^\top
=
\frac{1}{\sqrt{(F_x^i)^2+(F_y^i)^2+(F_z^i)^2}}
[F_x^i,F_y^i,F_z^i]^\top ,
\]

while the inverse depth is defined as

\[
\lambda^i
=
\frac{1}{\sqrt{(F_x^i)^2+(F_y^i)^2+(F_z^i)^2}} .
\]

The state vector is therefore chosen as

\begin{equation}
\label{EquationSfMState}
x=
\left[
u_x^1,u_y^1,u_z^1,\lambda^1,
\ldots,
u_x^K,u_y^K,u_z^K,\lambda^K
\right]^\top .
\end{equation}

The inverse-depth parametrization has been introduced in monocular Visual SLAM to improve the representation of feature uncertainty and to allow the initialization of landmarks with unknown depth \cite{Civera2008}. 

This representation is intentionally redundant. 
Indeed, the bearing vector does not contain independent quantities, since it satisfies the unit norm constraint. 
However, this choice leads to a particularly simple measurement model, because the monocular camera directly measures the bearing vectors associated with the observed feature points.
Moreover, the formulation does not introduce any additional sensing assumption.

It is worth emphasizing that, independently of the number of observed points \(K\), the absolute scale of the scene is not observable from monocular measurements alone. 
Consequently, the state representation above contains an intrinsically unobservable degree of freedom associated with the global scale factor. 

The measurement model is therefore given by

\begin{equation}
\label{EquationSfMOutput}
h(x)=
\left[
u_x^1,u_y^1,u_z^1,
\ldots,
u_x^K,u_y^K,u_z^K
\right]^\top .
\end{equation}

\subsection{Unknown-input system formulation}

For clarity, the dynamic equations are first derived for a single feature point, omitting the superscript \(i\).
The extension to \(K\) points is obtained by concatenating the corresponding dynamics.

Let

\[
V=
[V_x,V_y,V_z]^\top
\]

denote the camera linear velocity and

\[
\Omega=
[\Omega_x,\Omega_y,\Omega_z]^\top
\]

the camera angular velocity, both expressed in the camera frame.
The dynamics of the bearing vector are

{\small
\begin{align}
\label{EquationSfMDynamicsUx}
\dot{u}_x &=
\Omega_z u_y-\Omega_y u_z
-\lambda
\Big[
V_x-(u_xV_x+u_yV_y+u_zV_z)u_x
\Big],
\\[6pt]
\label{EquationSfMDynamicsUy}
\dot{u}_y &=
\Omega_x u_z-\Omega_z u_x
-\lambda
\Big[
V_y-(u_xV_x+u_yV_y+u_zV_z)u_y
\Big],
\\[6pt]
\label{EquationSfMDynamicsUz}
\dot{u}_z &=
\Omega_y u_x-\Omega_x u_y
-\lambda
\Big[
V_z-(u_xV_x+u_yV_y+u_zV_z)u_z
\Big].
\end{align}
}

The inverse depth dynamics is

\begin{equation}
\label{EquationSfMDynamicsLambda}
\dot{\lambda}
=
\lambda^2
(u_xV_x+u_yV_y+u_zV_z).
\end{equation}

Equations~\eqref{EquationSfMDynamicsUx}--\eqref{EquationSfMDynamicsLambda} can be directly expressed in the general form of
(\ref{EquationSystemDefinitionUIOsimp}).
The vector fields associated with the three linear velocity components are

\begin{equation}
\label{EquationSfMgV}
g^{V_x}
=
\begin{bmatrix}
-\lambda(1-u_x^2)\\
\lambda u_xu_y\\
\lambda u_xu_z\\
\lambda^2u_x
\end{bmatrix},
~
g^{V_y}
=
\begin{bmatrix}
\lambda u_xu_y\\
-\lambda(1-u_y^2)\\
\lambda u_yu_z\\
\lambda^2u_y
\end{bmatrix},
\end{equation}

\[
g^{V_z}
=
\begin{bmatrix}
\lambda u_xu_z\\
\lambda u_yu_z\\
-\lambda(1-u_z^2)\\
\lambda^2u_z
\end{bmatrix},
\]

while the vector fields associated with the angular velocity components are

\begin{equation}
\label{EquationSfMgOmegaX}
g^{\Omega_x}
=
\begin{bmatrix}
0\\
u_z\\
-u_y\\
0
\end{bmatrix},
~
g^{\Omega_y}
=
\begin{bmatrix}
-u_z\\
0\\
u_x\\
0
\end{bmatrix},
~
g^{\Omega_z}
=
\begin{bmatrix}
u_y\\
-u_x\\
0\\
0
\end{bmatrix}.
\end{equation}

For the complete SfM problem, the state vector is obtained by stacking the states associated with the \(K\) feature points, while the vector fields above are extended by placing the corresponding four-dimensional blocks in the components associated with each feature.

The resulting system always admits the representation
(\ref{EquationSystemDefinitionUIOsimp}).
The unknown inputs are selected among the six motion components

\[
\{V_x,V_y,V_z,\Omega_x,\Omega_y,\Omega_z\},
\]

depending on the available inertial measurements.
Each measured motion component is incorporated into the drift vector field \(g^0\), whereas each unmeasured component is treated as an unknown input multiplying its corresponding vector field.

For example, if all six motion components are treated as unknown inputs, no term contributes to the drift vector field. Hence, the drift vanishes and
\(
g^0(x)=\mathbf{0},
\) while the unknown-input vector fields are
\(
g^1=g^{V_x},\;
g^2=g^{V_y},\;
g^3=g^{V_z},\;
g^4=g^{\Omega_x},\;
g^5=g^{\Omega_y},\;
g^6=g^{\Omega_z}.
\)

If only \(\Omega_z\) is measured, then \(
g^0=\Omega_z\,g^{\Omega_z},\)
while the remaining unknown-input vector fields are
\(
g^1=g^{V_x},\;
g^2=g^{V_y},\;
g^3=g^{V_z},\;
g^4=g^{\Omega_x},\;
g^5=g^{\Omega_y}.
\)

Similarly, if the complete angular velocity is available, then
\(
g^0=
\Omega_x g^{\Omega_x}
+\Omega_y g^{\Omega_y}
+\Omega_z g^{\Omega_z},
\)
and only the three linear velocity components remain as unknown inputs,
\(
g^1=g^{V_x},\;
g^2=g^{V_y},\;
g^3=g^{V_z}.
\)

The remaining sensing configurations are obtained analogously by assigning every measured motion component to the drift vector field and every unmeasured component to the corresponding unknown-input vector field.

\subsection{Unknown-input observability analysis}
\label{SubSectionSfMObservability}

The observability properties of the SfM system are analyzed by applying Algorithm~\ref{AlgoFull} to the nonlinear system introduced above. 
The resulting observable codistribution \(\mathcal{O}\) provides a complete characterization of the state quantities that can be recovered from the available measurements in the presence of unknown motion components.

When \(K\geq4\), the algorithm shows that the state is observable up to the global scale factor even when all six camera motion components are treated as unknown inputs. 
Equivalently, the null space of the observable codistribution has dimension one, corresponding only to the inherent scale ambiguity of monocular vision. 
This result is consistent with classical Structure-from-Motion theory, where four or more point correspondences are sufficient for reconstruction in the absence of rotational measurements.

A different behavior occurs in the three-feature configuration. 
For \(K=3\), the system remains observable up to scale when at least one angular velocity component is measured, namely when a single-axis gyroscope is available. The corresponding minimal sensing configuration therefore consists of measuring a single angular velocity component while treating the remaining five camera motion components as unknown inputs. In this case, the observable codistribution again has a one-dimensional null space associated exclusively with the global scale ambiguity. Any additional motion measurements preserve this property.
Conversely, when no gyroscopic information is available and all six motion components are considered unknown inputs, the null space of \(\mathcal{O}\) has dimension larger than one (specifically, $\dim\left\{\OBS^\bot\right\}=3$). 
Therefore, additional ambiguities exist beyond the unavoidable scale factor, and the three-feature configuration is not sufficient for monocular reconstruction.

For \(K=2\), the observability analysis shows that complete rotational information is required in order to recover the state up to the intrinsic scale ambiguity. The corresponding minimal sensing configuration consists of measuring all three angular velocity components while treating the three linear velocity components as unknown inputs.

Finally, the analysis reveals an additional sensing configuration that is not usually emphasized in classical SfM formulations. 
If at least one component of the camera linear velocity is known, the state becomes fully observable, including the global scale factor, starting from \(K\geq3\). 
In this case, the scale ambiguity is removed because the additional velocity information provides an absolute metric constraint. The corresponding minimal sensing configuration consists of measuring a single component of the camera linear velocity, while treating the remaining five motion components as unknown inputs. Any additional motion measurements preserve full observability.
Moreover, full observability can also be achieved with only \(K=2\) point features, provided that no more than three motion components are treated as unknown inputs. The minimal sensing configurations correspond to measuring either one linear-velocity component and two angular-velocity components, two linear-velocity components and one angular-velocity component, or all three linear-velocity components. Any additional motion measurements preserve full observability.
Finally, full observability can also be achieved with a single point feature (\(K=1\)), provided that no more than one motion component is treated as an unknown input. The corresponding minimal sensing configuration therefore consists of measuring five of the six camera motion components.

These results are obtained directly from the unknown-input observability analysis and do not require any prior assumption on the evolution of the unmeasured camera motion components.
The complete symbolic computation of the observable codistribution is performed using the MATLAB implementation of Algorithm~\ref{AlgoFull} provided in \cite{TUIO}\footnote{As already pointed out, the state representation in (\ref{EquationSfMState}) is redundant, since each bearing vector satisfies the unit-norm constraint. Consequently, when using the MATLAB implementation of Algorithm~\ref{AlgoFull}, this constraint must also be enforced during the numerical initialization phase (implemented in the routine \texttt{set\_tool.m}), where random numerical values are assigned to the symbolic state variables for rank and determinant evaluations. More generally, whenever the algorithm is applied using a redundant state parametrization, all algebraic constraints among the state variables must be taken into account during this initialization. Otherwise, numerical rank or determinant computations may lead to incorrect conclusions.}.

\subsection{Construction of the UID-Induced Normal Form and EKF Implementation}
\label{SubSectionSfMUIDNormalForm}

The quantities generated by Algorithm~\ref{AlgoFull} are now used to construct the UID-induced normal form introduced in Section~\ref{SectionRecursiveEstimation}.
As discussed in Section~\ref{SubSectionSelectionEuristic}, when multiple admissible UID choices exist, the UIDs are selected by maximizing the absolute value of the determinant of the matrix \(\mu\) during the UID selection procedure.
An important property of the proposed SfM formulation is that the selected UIDs can be chosen directly among the available outputs.
This property considerably simplifies the practical implementation of the UID-induced normal form, since the quantities replacing the unknown inputs that affect the observable dynamics are directly available from the measurements.
Once the UIDs have been selected, the UID-induced normal form is constructed.

The recursive state estimator is then implemented using the EKF formulation described in Section~\ref{SubSectionRecursiveFiltering}.

The numerical evaluation of the proposed approach, including the minimal three-feature configuration with a single-axis gyroscope and the more general sensing configurations discussed in Section~\ref{SubSectionSfMObservability}, is presented in the following section.

\section{Experimental Validation}
\label{SectionExperimentalValidation}

This section experimentally validates the practical implications of the proposed UID-induced normal form through the Structure-from-Motion problem defined in Section~\ref{SectionSfM}. Specifically, the EKF formulation derived from the UID-induced normal form is evaluated on real-world data from the EuRoC MAV dataset~\cite{Burri2016EuRoC},

The experimental evaluation focuses on sensing configurations based on a monocular camera and angular velocity measurements. 
Although the employed inertial measurement unit also provides linear acceleration measurements, only gyroscopic information is considered in the experiments. 
Including accelerometer measurements would lead to a different class of SfM formulations, since linear accelerations are coupled with gravity, thereby introducing a dependence on the absolute orientation of the sensor. This coupling concerns the orientation components associated with the gravity direction, while the yaw component remains unobservable from accelerometer measurements alone.
While such configurations could also be investigated within the proposed framework, their analysis would require a dedicated treatment and would introduce a different application. 

The experiments investigate the sensing configurations predicted by the nonlinear unknown-input observability analysis presented in Section~\ref{SubSectionSfMObservability}. 
In particular, we consider configurations obtained by varying the number of observed point features and the available gyroscopic information, including the minimal three-feature case with a single-axis gyroscope and the purely vision-based case in which all six camera motion components are treated as unknown inputs.

\subsection{Dataset and Experimental Setup}
\label{SubSectionDatasetSetup}

The experiments are performed on the EuRoC MAV dataset~\cite{Burri2016EuRoC}, 
which provides synchronized camera images, inertial measurements, and ground-truth trajectories.
The dataset is used to evaluate the proposed estimator on real-world visual-inertial data while allowing a controlled modification of the available sensing information.

All the experiments reported in this paper are conducted on sequence \textit{V1\_01\_easy}. Unless otherwise stated, the estimator is evaluated over the same temporal interval, starting at timestamp $1403715394962142976$~ns and spanning approximately $7$~s. Using the same dataset sequence and time interval for all sensing configurations allows a direct comparison of the estimation performance under different unknown-input settings while keeping the observed scene and camera motion unchanged.

The visual measurements are obtained from the left monocular camera (\textit{cam0}) of the EuRoC MAV platform.
Point feature tracks are extracted from the image sequence and subsequently verified to ensure consistent point correspondences over the considered time interval.

To obtain bearing measurements from image features, the calibrated camera model provided with the dataset is taken into account.
In particular, the conversion from image pixel coordinates to unit bearing vectors in the camera frame is performed using the camera intrinsic parameters together with the radial-tangential distortion coefficients.
The raw image measurements are therefore undistorted according to the calibrated lens model before computing the corresponding bearing vectors.
This procedure ensures that the visual measurements used by the estimator are expressed consistently with the camera geometry and with the camera-to-body extrinsic calibration provided by the dataset.

Figure~\ref{FigSelectedPoints} shows the first image of the selected time interval (corresponding to timestamp $1403715394962142976$~ns) together with ten tracked point features. These features constitute the pool of visual landmarks employed in the experiments. Depending on the sensing configuration under consideration, only a subset of these features is used by the estimator, while the remaining ones are discarded. This enables a controlled experimental comparison between the minimal and non-minimal Structure-from-Motion configurations analyzed in this paper.

\begin{figure}[t]
\centering
\includegraphics[width=0.9\columnwidth]{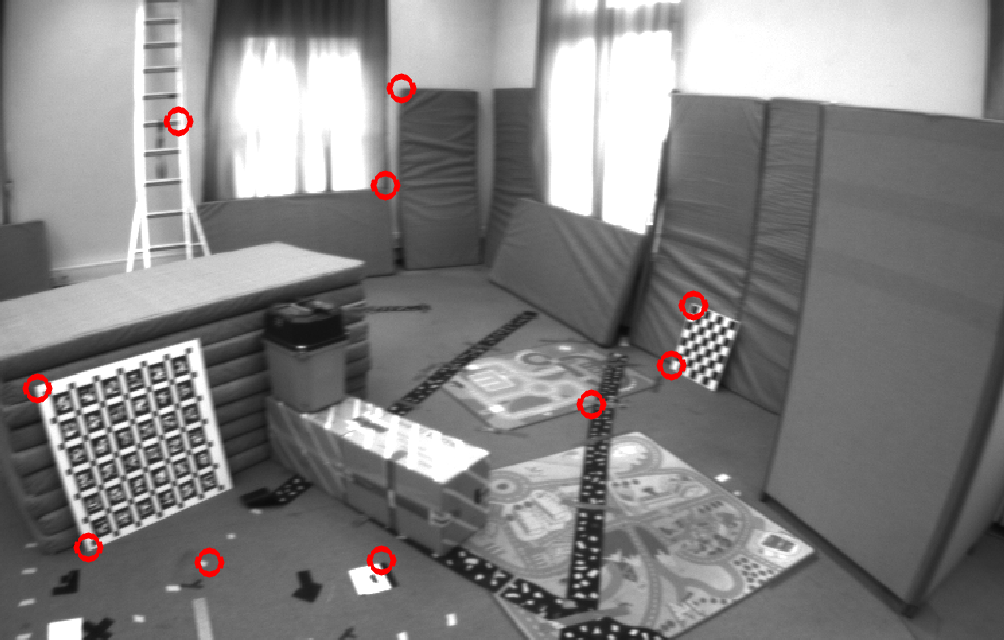}
\caption{First image of the selected EuRoC \textit{V1\_01\_easy} sequence (timestamp $1403715394962142976$~ns) with the ten tracked point features used in the experiments. Depending on the considered sensing configuration, only a subset of these features is employed by the estimator.}
\label{FigSelectedPoints}
\end{figure}

The ground-truth state associated with the selected point features is obtained by exploiting the ground-truth camera trajectory provided by the dataset.
In particular, the bearing measurements of each feature are available over multiple camera poses, while the corresponding camera poses in the global frame are known from the ground-truth trajectory.
By triangulating the feature observations across multiple images, the three-dimensional position of each point feature is accurately reconstructed in the global frame.
The use of multiple views rather than a single image pair improves the robustness and accuracy of the reconstructed feature positions.

Once the feature positions are available in the global frame, they are transformed into the camera frame at each time instant using the ground-truth camera pose.
The corresponding bearing vectors and inverse depths are then computed according to the state parametrization introduced in Section~\ref{SubSectionSfMState}.
This procedure provides a ground-truth reference for the state variables estimated by the EKF, namely the bearing directions and inverse depths of the observed point features.

The selection of the visual features is performed only to provide reliable measurements for the Structure-from-Motion formulation, since the objective of the experiments is to evaluate the proposed unknown-input estimation framework rather than feature detection or data association methods.

The inertial measurements are obtained from the onboard gyroscope.
According to the sensing configuration under evaluation, only a subset of the available angular velocity components is provided to the estimator, while the remaining motion components are treated as unknown inputs within the UID-induced normal form framework.

The ground-truth trajectory provided by the EuRoC MAV dataset is used to reconstruct the feature ground truth, to assess the estimation accuracy, and, in additional experiments, to generate camera-frame linear velocity measurements for configurations involving metric velocity information.

The estimator is initialized as follows. The bearing components of the state are initialized with the bearing measurements corresponding to the first image of the sequence, whereas all inverse-depth variables are initialized to $\lambda^i=1$. Hence, the initial feature depths are completely unknown to the estimator and are not initialized using ground-truth information.
All experiments are performed using the same experimental protocol described above.

\subsection{Evaluated Sensing Configurations}
\label{SubSectionEvaluatedConfigurations}

The considered experimental configurations are selected according to the observability analysis presented in Section~\ref{SubSectionSfMObservability}.
Each configuration corresponds to a different partition between measured motion components and unknown inputs.
Table~\ref{TabExperimentalConfigurations} summarizes the evaluated cases.
It is worth emphasizing that the term \emph{Unobservable} used for configuration C0 has a specific meaning in Table~\ref{TabExperimentalConfigurations}. As established by the observability analysis in Section~\ref{SubSectionSfMObservability}, all configurations that do not include linear velocity measurements remain unobservable because of the intrinsic global scale ambiguity of monocular Structure-from-Motion. This applies to configurations C1, C2, C3, C5, C6, and C8. However, in these cases the global scale factor is the \emph{only} unobservable degree of freedom, so that the three-dimensional structure and camera motion can still be reconstructed uniquely up to scale. By contrast, configuration C0 possesses additional unobservable directions beyond the global scale ambiguity, making Structure-from-Motion impossible even up to an unknown scale factor.

\begin{table}[t]
\centering
\caption{Experimental sensing configurations evaluated in this paper. Minimal configurations predicted by the observability analysis are highlighted in bold.}
\label{TabExperimentalConfigurations}
\renewcommand{\arraystretch}{1.15}
\begin{tabular}{c c c c p{1.8cm}}
\hline
\textbf{Case} &
\textbf{Point} &
\textbf{Gyro} &
\textbf{Linear} &
\textbf{Remarks} \\
 &
\textbf{Features} &
\textbf{Axes} &
\textbf{Velocity Axes} &
\\
\hline
C0 & 2 & 2 & 0 & Unobservable \\
\textbf{C1} & \textbf{3} & \textbf{1} & \textbf{0} & \textbf{Minimal} \\
C2 & 3 & 2 & 0 & --- \\
C3 & 3 & 3 & 0 & --- \\
\textbf{C4} & \textbf{3} & \textbf{0} & \textbf{1} & \textbf{Minimal (metric scale)} \\
\textbf{C5} & \textbf{4} & \textbf{0} & \textbf{0} & \textbf{Minimal (pure vision)} \\
C6 & 4 & 3 & 0 & --- \\
C7 & 4 & 3 & 1 & Metric scale \\
C8 & 8 & 0 & 0 & Pure vision \\
C9 & 8 & 0 & 1 & Metric scale \\
\hline
\end{tabular}
\end{table}

For each configuration, the EKF is implemented using the same UID-induced normal form framework.
Only the available measurements and the corresponding unknown-input partition are modified according to the considered sensing scenario.

The objective is not to optimize the estimator for a specific sensing configuration, but rather to verify whether the practical estimation behavior agrees with the structural observability properties predicted by the theoretical analysis.

\subsection{Experimental Results}
\label{SubSectionExperimentalResults}

The experimental results are evaluated with respect to the state representation introduced in Section~\ref{SubSectionSfMState}. 
Recall that the state consists, for each observed feature, of a unit bearing vector and the corresponding inverse depth. 
The bearing components are directly measured by the monocular camera and constitute the output of the system. 
Throughout all the experiments, their estimation remains very close to the measured values and to the corresponding ground truth. 
For this reason, no quantitative results are reported for the bearing components. 
A more challenging analysis would arise in the presence of image outliers, feature mismatches, or prolonged occlusions, where the visual measurements themselves become unreliable. 
Such issues, however, are outside the scope of this paper, whose objective is to validate the proposed UID-induced normal form as a structural framework for state estimation in the presence of unknown inputs. 
In the considered SfM application, the unknown inputs correspond to the camera motion components that are not directly measured, including linear and angular velocity components. 
Consequently, the experimental evaluation focuses exclusively on the estimation of the inverse-depth variables.

As discussed in Section~\ref{SubSectionSfMObservability}, the absolute values of the inverse depths are not observable in all the considered sensing configurations. 
Whenever metric scale is not observable, monocular Structure-from-Motion can only recover the scene geometry up to an unknown global scale factor. 
In terms of the adopted state parametrization, this implies that the physically meaningful quantities are not the individual inverse depths themselves, but their relative values. 
Accordingly, the estimation accuracy is evaluated through the ratios between the inverse depths associated with the observed point features.

For each sensing configuration, we compute over the entire experimental interval the mean error between the estimated inverse-depth ratios and the corresponding ratios obtained from the ground-truth state. 
Since all inverse-depth variables are initialized to the same value,
\[
\lambda^i=1,
\]
all estimated ratios are initially equal to one, whereas the true ratios generally differ significantly from unity. 
Consequently, the relative-depth error is initially large and progressively decreases as the estimator reconstructs the scene structure. 
The resulting error curves therefore directly illustrate the convergence of the proposed estimator toward the correct three-dimensional structure up to the intrinsic scale ambiguity predicted by the observability analysis.

\begin{table}[t]
\centering
\caption{Final estimation errors of the inverse-depth variables. The relative error evaluates the reconstructed scene geometry independently of the global scale factor, while the absolute error evaluates metric reconstruction.}
\label{TabFinalErrors}
\renewcommand{\arraystretch}{1.15}
\begin{tabular}{ccc}
\hline
\textbf{Case} &
\textbf{Mean Relative} &
\textbf{Mean Absolute} \\
&
\textbf{Inverse-Depth Error} &
\textbf{Inverse-Depth Error} \\
&
\textbf{(--) } &
\textbf{($\mathrm{m}^{-1}$)} \\
\hline
C0 &  0.9215 &  0.8069\\
C1 &  0.0612 &  0.4502\\
C2 &  0.0363 &  0.4387\\
C3 &  0.0088 &  0.4978\\
C4 &  0.1058 & 0.0672\\
C5 &  0.0673 &  0.3095\\
C6 &  0.0045 &  0.3464\\
C7 &  0.0108&  0.0182 \\
C8 &  0.0038&  0.2303\\
C9 &  0.0025&  0.0014\\
\hline
\end{tabular}
\end{table}


For completeness, the final estimation errors of the inverse-depth variables are reported for all the considered sensing configurations. 
The results are summarized in Table~\ref{TabFinalErrors}, where the relative inverse-depth error evaluates the quality of the reconstructed scene geometry independently of the global scale factor, while the absolute inverse-depth error evaluates the accuracy of the metric reconstruction whenever the scale is observable.

The experimental results fully agree with the structural predictions provided by the nonlinear unknown-input observability analysis. 
As shown in Table~\ref{TabFinalErrors}, configurations without metric scale information exhibit relatively large absolute inverse-depth errors. 
This behavior is expected because, in these cases, the estimated scene structure is defined only up to an arbitrary global scale factor. 
Therefore, the absolute values of the inverse depths may differ significantly from the ground truth even when the reconstructed geometry is correct. 
Conversely, the relative inverse-depth errors remain small for all configurations predicted to be observable up to scale.

\begin{figure*}[t]
\centering

\includegraphics[width=0.85\textwidth]{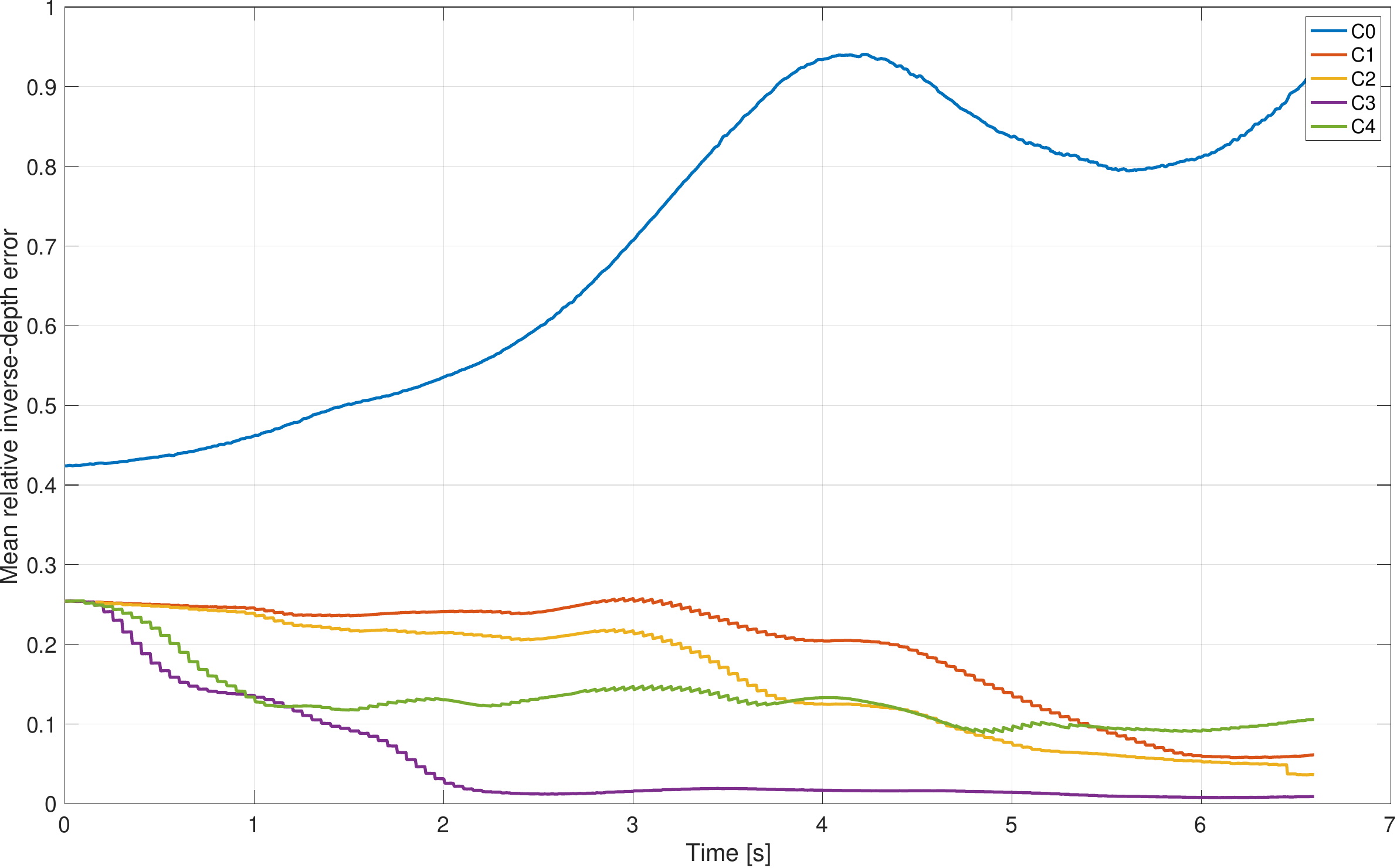}

\vspace{0.3cm}

\includegraphics[width=0.85\textwidth]{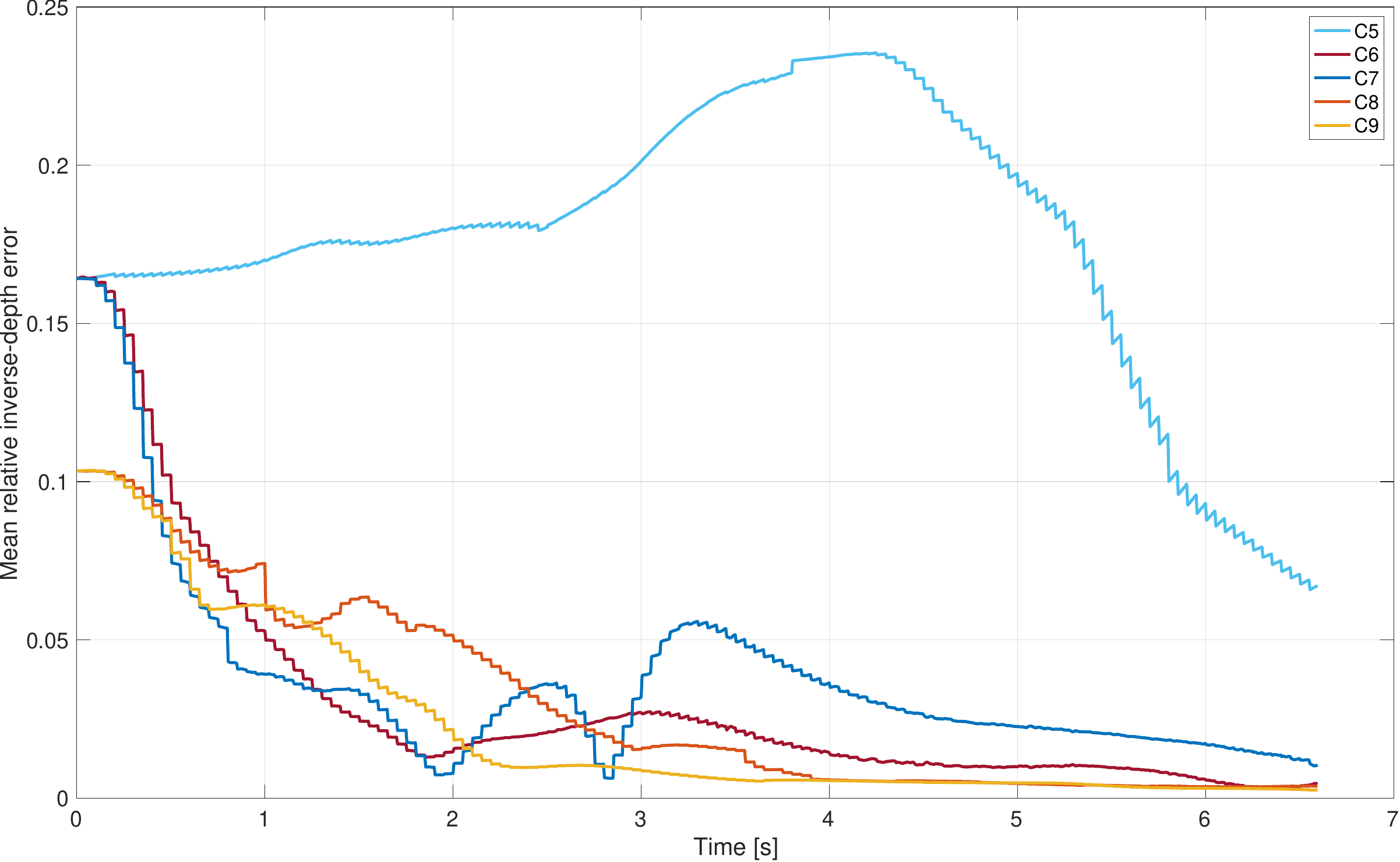}

\caption{Experimental validation of the UID-induced normal form on the EuRoC MAV dataset. 
The plots report the mean relative inverse-depth error over the considered point features during the estimation interval. 
The first plot compares configurations C0--C4, including the unobservable configuration C0 and the minimal configurations involving three point features. 
The second plot reports configurations C5--C9, corresponding to additional observable sensing configurations with increased visual or inertial information. 
The results show that all configurations predicted as observable up to the global scale factor achieve convergence of the relative inverse-depth estimates, whereas the unobservable configuration C0 exhibits a persistent estimation error.}

\label{FigLambdaRelativeError}
\end{figure*}


The convergence behavior of the relative inverse-depth estimates is further illustrated in Figure~\ref{FigLambdaRelativeError}. 
The reported curves show that all configurations predicted by the observability analysis to be observable up to the global scale factor achieve convergence toward the correct relative scene geometry. 
In contrast, configuration C0, which possesses additional unobservable directions beyond the intrinsic scale ambiguity, fails to reconstruct the structure and maintains a significantly larger error throughout the estimation interval.

When metric information is provided through linear velocity measurements, the global scale factor becomes observable. 
As a consequence, configurations C4, C7, and C9 achieve accurate estimation not only of the relative inverse depths but also of their absolute values, as reflected by the significantly smaller absolute inverse-depth errors reported in Table~\ref{TabFinalErrors}. 
These results demonstrate that the UID-induced normal form correctly captures the structural decomposition revealed by unknown-input observability and enables the EKF formulation to achieve the estimation behavior predicted by the theoretical analysis.

\section{Conclusion}\label{SectionConclusion}

This paper provided the first general structural solution to state estimation for nonlinear systems driven by unknown inputs.
Building upon the nonlinear unknown-input observability theory developed in \cite{TAC19,IF22,TAC_UIO}, we showed that every nonlinear system with unknown inputs admits a canonical representation, referred to as the \emph{UID-induced normal form}.
For any admissible selection of unknown input identifiers (UIDs), this representation is uniquely determined by the selected UIDs and reveals the intrinsic decomposition of the unknown-input information affecting the observable dynamics.
The construction does not require any model, stochastic assumption, or dynamical description of the unknown inputs.

The UID-induced normal form provides a complete structural decomposition of the information carried by the unknown inputs.
Specifically, it separates unknown-input directions that have no influence on the observable dynamics and are therefore structurally decoupled from state estimation, from the complementary unknown-input information that affects the observable dynamics and is represented through recursively generated observable quantities.
Although different admissible UID selections may lead to different coordinate representations of the unknown-input space, they are equivalent with respect to the observable dynamics.
As a consequence, unknown-input decoupling and unknown-input reconstruction are no longer separate problems, but emerge as two complementary aspects of the same observability-based decomposition.
The proposed framework therefore establishes a unified structural foundation for state estimation in the presence of unknown inputs.

The practical relevance of this result was demonstrated through a Structure-from-Motion application.
The UID-induced normal form enabled the construction of an EKF-based recursive estimator for a previously unexplored minimal sensing configuration, where three-dimensional structure and camera motion are recovered, up to the intrinsic global scale ambiguity, from only three point features and a single-axis gyroscope.
Experimental results on real-world data validated the proposed approach and confirmed the feasibility of the identified minimal configuration.
Moreover, the analysis of additional sensing configurations showed the consistency of the proposed framework with the observability properties predicted by classical SfM theory.

Beyond the considered application, the UID-induced normal form provides a general structural representation for nonlinear systems affected by unknown inputs.
Since it is derived from the observability structure of the system, it offers a principled foundation for the systematic development of estimation methods in a broad range of applications, including observer design, filtering, fault diagnosis, resilient estimation, and sensor fusion.
More generally, this work suggests that unknown-input estimation can be approached not through additional assumptions on the unknown inputs, but by exploiting the intrinsic observability structure of the system.

\appendix

\section{Proof of Theorem \ref{PropertyIrrilevance}}\label{AppendixIrrilevance}

We first establish the following preliminary result.

\begin{lm}\label{Lemma1}
Consider the system expressed in the UID-induced normal form~\eqref{EquationUIDNF}.
Then, for every smooth function $\phi=\phi(x,t)$ whose differential belongs to the
observable codistribution, i.e., $\diff\phi\in\OBS$, the following property holds:
\[
\frac{\partial \phi}{\partial x} \cdot \giat^k = 0, \qquad \forall k>m_H.
\]
\end{lm}

\begin{proof}
By assumption, the differential of $\phi=\phi(x,t)$ belongs to the observable
codistribution $\OBS$. This codistribution is generated by Algorithm~\ref{AlgoFull}.
The algorithm terminates either through Line~\ref{AlgoLineConditionOmegaStarReturn}
or through Line~\ref{AlgoLineFINALSTEP}.

The case in which the algorithm terminates through Line~\ref{AlgoLineFINALSTEP}
can be excluded from further consideration, since in this case $m_H=\mw$,
and hence there exists no index $k>m_H$.

Conversely, suppose that the algorithm terminates through
Line~\ref{AlgoLineConditionOmegaStarReturn}. In this case, we have
\[
\OBS=\Omega_*,
\]
where $\Omega_*$ is the codistribution defined in Line~\ref{AlgoLineOmegaStar}.
The termination condition in Line~\ref{AlgoLineConditionOmegaStarReturn} implies
that the condition in Line~\ref{AlgoLineConditionOmegaStarIf} is not satisfied.
Equivalently, we have
\[
\giat^k\in\Omega_*^\bot, \qquad \forall k>m_H.
\]
Since $\OBS=\Omega_*$, the above relation can be rewritten as
\[
\giat^k\in\OBS^\bot, \qquad \forall k>m_H.
\]
Finally, because $\diff\phi\in\OBS$, it follows directly from the definition of
the annihilator that
\[
\frac{\partial \phi}{\partial x}\cdot\giat^k=0,
\qquad \forall k>m_H,
\]
which proves the claim.
$\blacktriangleleft$
\end{proof}

\vspace{.2cm}

We are now ready to prove Theorem~\ref{PropertyIrrilevance}.

\begin{proof}
The result follows by direct computation. Consider the time derivative of
$\phi=\phi(x,t)$:
\[
\dot{\phi}
=
\frac{\partial \phi}{\partial x}\dot{x}
+
\frac{\partial \phi}{\partial t}.
\]
Substituting the dynamics of the UID-induced normal form~\eqref{EquationUIDNF},
we obtain
\[
\dot{\phi}
=
\frac{\partial \phi}{\partial x}\cdot
\left(
\giat^0
+
\sum_{j=1}^{m_H}\giat^j v_j
+
\sum_{k=m_H+1}^{\mw}\giat^k w_k
\right)
+
\frac{\partial \phi}{\partial t},
\]
which can be rewritten as
\[
\dot{\phi}
=
\dt{\Li}_{\giat^0}\phi
+
\sum_{j=1}^{m_H}
\Li_{\giat^j}\phi\, v_j
+
\frac{\partial \phi}{\partial x}\cdot
\left(
\sum_{k=m_H+1}^{\mw}\giat^k w_k
\right).
\]

By Lemma~\ref{Lemma1}, we have
\[
\frac{\partial \phi}{\partial x}\cdot\giat^k=0,
\qquad \forall k>m_H.
\]
Consequently, the last term in the previous expression vanishes, yielding
\[
\dot{\phi}
=
\dt{\Li}_{\giat^0}\phi
+
\sum_{j=1}^{m_H}
\Li_{\giat^j}\phi\, v_j .
\]
Therefore, the time evolution of any function whose differential belongs to the
observable codistribution is independent of the input components $w_k$,
for all $k>m_H$, which proves the theorem.
$\blacktriangleleft$
\end{proof}
\section{Proof of Theorem \ref{PropertyReconstructability}}\label{AppendixReconstructability}

\begin{proof}
The result follows directly from the construction of $v_1,\ldots,v_{m_H}$. From~\eqref{Equationvj}, each component $v_j$ satisfies
\[
v_j = \dot{\tih}_j, \qquad j=1,\ldots,m_H.
\]
In addition, each $\tih_j$ is an observable quantity, i.e., $\diff \tih_j \in \OBS$. Therefore, $\tih_j$ can be computed from the system outputs and a finite number of their time derivatives. As a consequence, its time derivative $v_j$ is also reconstructible from the outputs and their derivatives. $\blacktriangleleft$
\end{proof}

\end{document}